\documentclass[reqno, 11pt]{amsart}

\newcommand{\tensor}{\otimes}

\newcommand{\op}{\mathcal}

\newcommand{\cdc}{,\dots,}

\newcommand{\fm}{\mathsf{fM}}
\newcommand{\kdb}{k_\Delta\langle\beta_\ast\rangle}

\usepackage{amsmath}%
\usepackage{amsthm}
\usepackage{amsfonts}%
\usepackage{amssymb}%
\usepackage{graphicx}
\usepackage{xy,amsthm,enumerate,xypic,array}
\usepackage{epstopdf,multirow}  
\input{xy}
\xyoption{all}

\numberwithin{equation}{section}

\newtheorem{theorem}{Theorem}[section]
\theoremstyle{plain}

\newtheorem*{nonumbertheorem}{Theorem}

\newtheorem{corollary}[theorem]{Corollary}
\newtheorem{lemma}[theorem]{Lemma}

\theoremstyle{definition}
\newtheorem{definition}[theorem]{Definition}

\newtheorem{remark}[theorem]{Remark}

\allowdisplaybreaks[2]

\addtocounter{MaxMatrixCols}{2}

\begin{document}

\title{Hyper-commutative algebras and cyclic cohomology}
\author{Benjamin C. Ward}



\begin{abstract}

	This paper introduces a chain model for the Deligne-Mumford operad formed by homotopically trivializing the circle in a chain model for the framed little disks.  We then show that under degeneration of the Hochschild to cyclic cohomology spectral sequence, a known action of the framed little disks on Hochschild cochains lifts to an action of this new chain model.  We thus establish homotopy hyper-commutative algebra structures on both Hochschild and cyclic cochain complexes, and we interpret the gravity brackets on cyclic cohomology as obstructions to degeneration of this spectral sequence.  Our results are given in the language of deformation complexes of cyclic operads.

\end{abstract}

\maketitle
\section*{Introduction}

A differential graded Batalin-Vilkovisky (BV) algebra enhanced with a homotopy trivialization of the $\Delta$-operator is equivalent to a hyper-commutative (HyCom) algebra \cite{DCV,KMS,DC}.  This relationship may be described in the language of operads, where BV and HyCom algebras are represented respectively by the homology operads of genus 0 moduli spaces of surfaces with boundary \cite{G94}, and by the Deligne-Mumford compactification of the moduli space of surfaces with punctures \cite{GetMod}.  In practice, BV algebras often arise as the homology or cohomology of a geometric, topological or algebraic object, and the chain level structure can only be expected to be BV up-to-homotopy.  For example, this is the case when studying Hochschild cochain operations via the cyclic Deligne conjecture \cite{KCyclic}.  More generally, this is the case when considering the deformation complex of a cyclic operad $\op{O}$ with $A_\infty$ multiplication $\mu$.  We denote such a deformation complex $CH^\ast(\op{O}, \mu)$.

On the other hand, the results of \cite{Ward2} show that the complex of cyclic invariants associated to such data carries a compatible structure of an algebra over a model of the {\it open} moduli space of punctured Riemann spheres.  This complex of invariants is a generalization of Connes' $C^\ast_\lambda$-complex and will be denoted $C^\ast_\lambda(\op{O},\mu)$.  Its cohomology $HC^\ast(\op{O},\mu)$ generalizes the notion of the cyclic cohomology of a cyclic $k$-module.  It is natural to ask for conditions under which this action of the open moduli space lifts to an action of an operad of chains on the Deligne-Mumford compactification.

In a BV algebra, the $\Delta$-operator corresponds to an action of the circle at a boundary component.  In the homotopy theory of $S^1$-spaces, trivialization of the circle action corresponds to degeneration of the Hochschild to cyclic (co)homology spectral sequence arising from the associated cyclic object.  Thus it is reasonable to expect that an analog of this degeneration will permit such a lifting.  We prove the following result (see Theorem $\ref{algthm}$):

\begin{nonumbertheorem}  Let $\mu\colon \op{A}_\infty\to\op{O}$ be a map of cyclic operads.  Let $CH^\ast(\op{O},\mu)$ and $C^\ast_\lambda(\op{O},\mu)$ be the associated deformation and cyclic deformation complexes of $\mu$.  If the morphism $\mu$ is cyclically degenerate (Definition $\ref{degdef2}$) then the homotopy BV algebra structure on $CH^\ast(\op{O},\mu)$ lifts to a compatible homotopy hyper-commutative algebra.  Moreover, $C^\ast_\lambda(\op{O},\mu)$ carries the structure of a $\op{H}y\op{C}om_\infty$-algebra for which the inclusion $C^\ast_\lambda(\op{O},\mu)\to CH^\ast(\op{O},\mu)$ extends to an $\infty$-morphism of $\op{H}y\op{C}om_\infty$-algebras.
\end{nonumbertheorem}

Examples of the complexes $CH^\ast(\op{O},\mu)$ and $C^\ast_\lambda(\op{O},\mu)$ include Hochschild and cyclic cochain complexes of Frobenius or cyclic $A_\infty$-algebras/categories, singular and equivariant cochains of $S^1$-spaces, as well as complexes computing string topology; homology and $S^1$-equivariant homology of the loop space of a closed oriented manifold, see \cite{Ward2} for these and other examples.

Formulating precise conditions under which the Hochschild to cyclic spectral sequence degenerates is a question of active study \cite{KS2}.  This question is of particular interest when studying homological mirror symmetry and categorical models of quantum cohomology \cite{BK,Manin}, which provided an expectation that the known BV/gravity structure should lift under degeneration.  Our result further says that, under degeneration and if the chain level structure can be seen as cyclic, the lifting can be performed at the chain level.  Pursuing this chain level structure in such geometric examples, with an eye toward a chain level lift of Gromov-Witten invariants, should be an interesting avenue for future study.

In proving this result we start with a chain model for the framed little disks given in \cite{Ward1} which we denote $\fm$.  Alternatively we could use cellular chains on Cacti \cite{Vor}, \cite{KCacti}.  These chain models have $\Delta^2=0$ and thus are susceptible to the language of mixed complexes.  Combining the constructions of \cite{KMS} and formality of the Deligne-Mumford (cyclic) operad \cite{GNPR} it is straight forward to prove:

\begin{nonumbertheorem}  Homotopically trivializing the $\Delta$-operator in $\fm$ gives a chain model for the hyper-commutative operad.
\end{nonumbertheorem}

This chain model is denoted $\fm_{hS^1}$ and it is, to an extent, combinatorially tractable.  For example, the fundamental class of $\overline{\op{M}}_4$ is represented by a union of fourteen 2-cells.  The combinatorics of $\fm_{hS^1}$ are based on planar trees with three types of vertices, see Section $\ref{combsec}$.  The following table summarizes the chain models for moduli spaces which act on such deformation complexes.  (Note that the assumption that $\op{O}$ is cyclic is not needed in the first row).

\begin{center}
\begin{tabular}{c|c|c|c|c}
{\bf Moduli }				      &			{\bf Algebras'} 	& {\bf Chain }					   &  {\bf Acts on /}                    &   {\bf Combin-} \\
{\bf Space}               &     {\bf  name  }    &     {\bf  Model }       &  {\bf theorem}                      &    {\bf atorics}    \\ \hline
points in           &   \multirow{2}{*}{Gersten-}        &   \multirow{2}{*}{$\mathsf{M}$  of \cite{KS}}             &  $CH^\ast(\op{O}, \mu)$.      & black and white   \\
 the plane:         &   \multirow{2}{*}{haber}    &    \multirow{2}{*}{(minimal operad)}      & $A_\infty$ Deligne-            &(or b/w)  planar            \\
 $\op{D}_2$         &          &                             &  conjecture               & rooted trees \\ 
\hline 
surfaces with                  & \multirow{3}{*}{BV} &  \multirow{2}{*}{$\mathsf{fM}$}  &  $CH^\ast(\op{O}, \mu)$.   &  b/w planar       \\
boundary:                             &                     &    \multirow{2}{*}{of \cite{Ward1}}         & cyclic $A_\infty$-         &  rooted trees  \\  
$f\op{D}_2\sim\widehat{\op{M}}_\ast$  &                     &          & Deligne conj.              &  with spines   \\ 
\hline 
surfaces with        & \multirow{3}{*}{gravity}      & \multirow{2}{*}{$\mathsf{M}_\circlearrowright$}  &  $C^\ast_\lambda(\op{O},\mu)$   &   b/w planar \\ 
punctures:           &                      &  \multirow{2}{*}{of \cite{Ward2}}             & $S^1$-equivariant &    non-rooted  \\  
$\op{M}_\ast$				 &                      &                               & Deligne Conj. &    trees\\
\hline 
Deligne-Mumford           & HyCom            &    \multirow{2}{*}{$\fm_{hS^1}$}     &      $CH^\ast(\op{O}, \mu)$.              & b/w/gray    \\
compactification:         & or formal        & \multirow{2}{*}{Theorem $\ref{chainmodelthm}$}  &      under degen.           & planar rooted  \\  
$\overline{\op{M}}_\ast$	& Frob. mfd        &                     &      Theorem $\ref{algthm}$           &  trees w/ spines 
\end{tabular}
\end{center}
\bigskip

In this paper we assume familiarity with operads and cyclic operads, references for which include \cite{GeK1,MSS,LV}.  We also assume familiarity with moduli space operads in genus $0$ and their homology, references for which include \cite{G94,GetMod,KSV}.
\subsection*{Acknowledgment}  This paper was written at the Simons Center for Geometry and Physics, whose support I gratefully acknowledge.  Thanks are due to Gabriel C. Drummond-Cole, Yuan Gao, Ralph Kaufmann, Dennis Sullivan, and Bruno Vallette for helpful discussions along the way.  Many thanks are also due to the anonymous referee for many helpful improvements. 
\tableofcontents

\section{The model categorical framework.}  

Let $\op{O}ps$ be the category of reduced operads valued in the category of differential graded (dg) vector spaces over a field $k$ of characteristic $0$.  Reduced means we restrict attention to arities $n\geq 1$.  The category $\op{O}ps$ is a model category such that forgetting to symmetric sequences creates weak equivalences and fibrations \cite{BMo}.  The reduced assumption will be needed to ensure that this model category is left proper: a pushout of a weak equivalence along a cofibration is a weak equivalence (see \cite{BB} Theorem 0.1). 

We consider an associative algebra to be an operad concentrated in arity $1$.  Here we may consider either unital associative algebras or operads without units.  For an associative algebra $A$ we define $A$-$\op{O}$ps to be the undercategory $A\searrow \op{O}ps$.  Note we often abuse notation by writing eg $\op{Q}\in A$-$\op{O}ps$ and not $A\stackrel{\alpha}\to\op{Q} \in A$-$\op{O}ps$.  Note that the undercategory of a left proper model category is a left proper model category where forgetting to the original category creates all three classes of distinguished morphisms. 

Let $f\colon A\to B$ be a morphism between dg associative algebras $A$ and $B$.  We define $\Gamma_f\colon B$-$\op{O}ps\to A$-$\op{O}ps$ to be the functor induced by composition with $f$.  The subscript notation will be suppressed when appropriate.  Note that $\Gamma$ preserves weak equivalences and fibrations and since $\Gamma$ is a right adjoint, it is a right Quillen functor.  The left adjoint $L$ can be realized as a left Kan extension or as a pushout of operads $\op{Q}\leftarrow A \rightarrow B$.

\begin{lemma}\label{welem}  If $f$ is a quasi-isomorphism, the Quillen adjunction $(L,\Gamma)$ is a Quillen equivalence.
\end{lemma}
\begin{proof}  If $A\to\op{O}$ is a cofibrant object in $A$-$\op{O}ps$ then the map $A\to\op{O}$ is a cofibration in $\op{O}ps$.  Using left properness of this category we know the pushout of $f$ along this cofibration is a weak equivalence $\op{O}\stackrel{\sim}\to L\op{O}$ in $\op{O}ps$.  In particular, $\op{O}\to \Gamma L\op{O}$ is a weak equivalence in $A$-$\op{O}ps$.  This along with the fact that $\Gamma$ reflects weak equivalences proves the claim via the standard theory (see \cite{Hovey} Cor. 1.3.16).
\end{proof}

We now consider the special case of an inclusion $\iota\colon A\hookrightarrow A\oplus B$ between dg associative algebras $A$ and $A\oplus B$.  We define $T_\iota\colon A$-$\op{O}ps\to (A\oplus B)$-$\op{O}ps$ to be the trivial extension by $0$.  It is again immediately clear from the model structure on the undercategories that $T$ preserves weak equivalences and fibrations and since $T$ is a right adjoint, it is a right Quillen functor.  We define its left adjoint $\Phi$ on a morphism $\epsilon\colon A\oplus B \to \op{P}$ by taking the quotient $\op{P}/(\epsilon(B))$.

In particular, for such an inclusion we have a pair of adjunctions:

\begin{equation*}
\xymatrix{ A\text{-}\op{O}ps  \ar@<-1.5ex>@{.>}[rr]_L  \ar@<1.5ex>[rr]^T   &  &  \ar@/^1pc/@<1.5ex>[ll]^\Gamma  \ar@/_1pc/@<-1.5ex>@{.>}[ll]_\Phi  (A\oplus B)\text{-}\op{O}ps}
\end{equation*}
We will typically be concerned with the inclusions $k\hookrightarrow H^\ast(S^1)$ and $k\hookrightarrow \Omega H_\ast(BS^1)$.

\begin{remark}  We may view the categories $A$-$\op{O}ps$ etc. as categories of monoidal functors \cite{KW}, in which case these adjunctions suggest a formal analogy with Verdier duality which could be further explored.
\end{remark}

The functors $T$ and $\Gamma$ preserve weak equivalences between fibrant objects (and hence all objects).  The functors $L$ and $\Phi$ preserve weak equivalences between cofibrant objects, and so cofibrant replacement will yield a well defined functor on the homotopy category.  Let $\tilde{\Phi}(\op{Q}):=\Phi(q\op{Q})$ and similarly for $L$, where $q$ means a choice of cofibrant replacement {\it in the undercategory}.  Then we have adjunctions $(\tilde{L}_f, \Gamma_f)$ and $(\tilde{\Phi}_\iota, T_\iota)$ between the respective homotopy categories corresponding to any morphism $f$ and any inclusion $\iota$.  We may now record the following technical lemma for future use.

\begin{lemma}\label{welem2}  Let $A\stackrel{\iota}\hookrightarrow A\oplus B \stackrel{f}\longrightarrow A\oplus C$ be a sequence of morphisms of associative dg algebras where $f$ is of the form $f=id_A\oplus\bar{f}$, and where $\iota$ is the inclusion.  If $f$ is a weak equivalence then for every $\op{Q}\in (A\oplus C)$-$\op{O}ps$ there is a zig-zag of weak equivalences of $A$-$\op{O}ps$ connecting
\begin{equation*}
\tilde{\Phi}_{f\circ \iota}(\op{Q})\sim \tilde{\Phi}_\iota(\Gamma\op{Q})
\end{equation*}
\end{lemma}
\begin{proof}
We continue to write $q$ for cofibrant replacement in any of these undercategories.

Since $(L_f,\Gamma_f)$ is a Quillen equivalence we know that the composite $Lq\Gamma q\op{Q}\to L\Gamma q\op{Q}\to q\op{Q}$ is a weak equivalence.  Moreover $L$ preserves cofibrant objects (since cofibrations are closed under pushout) and so this composite is a weak equivalence between cofibrant objects.  Hence the induced map
\begin{equation*}
\Phi_{f\circ \iota}( Lq\Gamma q\op{Q} )\stackrel{\sim}\to\Phi_{f\circ \iota}( q\op{Q} )
\end{equation*}
is a weak equivalence.  

The assumptions here imply that $\Gamma_f\circ T_{f\circ \iota}=T_\iota$ and so by adjointness, we know $\Phi_{f\circ \iota}\circ L_f \cong \Phi_\iota$.  Therefore $\Phi_{f\circ \iota}( Lq\Gamma q\op{Q} )\cong \Phi_\iota(q\Gamma q\op{Q})$ and so we have
\begin{equation*}
\tilde{\Phi}_\iota(\Gamma\op{Q})=\Phi(q\Gamma\op{Q})\stackrel{\sim}\longleftarrow \Phi_\iota(q\Gamma q\op{Q})  \cong \Phi_{f\circ \iota}( Lq\Gamma q\op{Q} )\stackrel{\sim}\to\Phi_{f\circ \iota}( q\op{Q} )=\tilde{\Phi}_{f\circ \iota}(\op{Q})
\end{equation*}
\end{proof}

\section{Trivializing $\Delta$.}  In this section we revisit the literature on mixed complexes, BV algebras, and trivializing $\Delta$ to extract what will be needed.  We primarily follow \cite{KMS}, with influence from \cite{DCV} and \cite{DSV}.

A mixed complex $(A,d,\Delta)$ is a chain complex $(A,\Delta)$ and a cochain complex $(A,d)$ such that $d\Delta+\Delta d=0$.  Here the degrees have been chosen to be consistent with the example of Hochschild cohomology, but this is merely our convention.  Given a mixed complex, let $(End_A,\partial)$ be the cochain complex of endomorphisms of $(A,d)$.  The cycle $\Delta$ is homologically trivial if there is a $\beta_1$ of degree $-2$ such that $\partial(\beta_1)=\Delta$.  Given such a $\beta_1$ we then encounter the cycle $\beta_1\circ\Delta$.  This cycle will be homologically trivial if there is a $\beta_2$ of degree $-4$ such that $\partial(\beta_2)=\beta_1\circ\Delta$.  Given such a $\beta_2$ we encounter the cycle $\beta_2\circ\Delta$ and so on.  We are thus led to the following definition.

\begin{definition}\label{degdef}  Let $(A,d,\Delta)$ be a mixed complex.  A trivialization of $\Delta$ is a sequence $\{\beta_i\}_{i\geq 0}\in End_{(A,d)}$ with $\beta_0=id_A$ such that $\partial(\beta_i)=\beta_{i-1}\circ\Delta$.  In particular, $|\beta_i|=-2i$.
\end{definition}

We write $(\kdb,\partial)$ for the dg associative algebra which encodes the unary operations on a trivialized mixed complex.  Explicitly, $\kdb:=k[\Delta]\langle \beta_1,\beta_2,...\rangle$, the free graded associative algebra on $\{\beta_i\}$ over the graded commutative algebra $k[\Delta]$.  The differential takes $\partial(\beta_i)=\beta_{i-1}\Delta$ and $\partial(\Delta)=0$, and the degrees of elements are as above.  In particular, $\Delta$ being of odd degree implies $\Delta^2=0$.   This complex may be interpreted as a non-commutative analog of $ES^1$; see Section $\ref{combsec}$.  The operations $\partial(\beta_i)$ correspond to differentials in the Hochschild to cyclic cohomology spectral sequence;  see Section $\ref{cycsec}$.

We write $A[[z]]:=A\tensor k[[z]]$, meaning the completed tensor product.
\begin{lemma}\label{trivlem}  Let $z$ be a variable of degree $2$ and let $(A,d,\Delta)$ be a mixed complex.  A trivialization of $\Delta$ is equivalent to a $z$-linear isomorphism of complexes $(A[[z]], d+z\Delta) \cong (A[[z]], d)$.
\end{lemma}
\begin{proof}  If $\{\beta_i\}$ are such a trivialization of $\Delta$, define a $z$-linear map $F\colon (A\tensor k[[z]], d+z\Delta) \to (A\tensor k[[z]], d)$ by $F(a)=\sum_i\beta_i(a)z^i$ for $a\in A$.  This is a dg map and its leading term is invertible.  Conversely given such an isomorphism we may extract the sequence $\{\beta_i(a)\}$ as the coefficients of $F(a)$.
\end{proof}

\begin{remark} The papers \cite{KMS}, \cite{DCV}, \cite{DSV} consider various equivalent forms of the data in Definition $\ref{degdef}$.  For example \cite{KMS} consider exponential coordinates for a trivialization, i.e. elements $\phi_i$ in the algebra $\kdb$ related to $\beta_i$ via $\beta_1=\phi_1$, $\beta_2=\phi_2+\phi_1^2/2$, and $\beta_3=\phi_3+(\phi_2\phi_1+\phi_1\phi_2)/2+\phi_1^3/6$, etc.  Note that in \cite{KMS}, the notation $\Phi_i$ is used in place of our $\beta_i$.  This data is equivalent to what \cite{DCV} call `Hodge-to-de Rham degeneration data', see Theorem 2.1 \cite{DSV}.  Note this data is weaker than the classical $d\bar{d}$-condition of \cite{DGMS}, which will not in general be satisfied in the examples we consider.
\end{remark}

\subsection{The KMS model for $\tilde{\Phi}$.}  In this section we consider the adjunction $(\Phi,T)$ with respect to the inclusion $k\to H^\ast(S^1)$. We further use the shorthand notation $S^1$-$\op{O}ps:=H^\ast(S^1)$-$\op{O}ps$, so we have $\Phi\colon S^1$-$\op{O}ps\leftrightarrows\op{O}ps\colon T$.

If $\op{Q}$ is an $S^1$-operad with operator $\Delta_\op{Q}$ we define an $S^1$-operad $W(\op{Q})$ by:
\begin{equation}
W(\op{Q}):=(\op{Q}\star \kdb, d_{W(\op{Q})}:=d_\op{Q}+\partial_\Delta-\partial_{\Delta_\op{Q}})
\end{equation}
Here $\star$ means free product,  $\partial_{\Delta_{\op{Q}}}(\beta_i)=\beta_{i-1}\Delta_\op{Q}$ and is zero on other generators, and similarly for $\partial_\Delta$.  This is viewed as an $S^1$-operad via $\Delta$ (not $\Delta_\op{Q})$.  We define $\Phi_{KMS}(\op{Q}):=\Phi(W(\op{Q}))$.  Viewing $\kdb$ as a non-commutative analog of $ES^1$, one expects that quotienting by the $\Delta$ action in $W(\op{Q})$ will be homotopy invariant.  Indeed \cite{KMS} shows:

\begin{theorem}\cite{KMS} As derived functors $\Phi_{KMS}\cong\tilde{\Phi}$.
\end{theorem}

\begin{corollary}\label{KMScor}  Suppose $\op{Q}$ is an $S^1$-operad and suppose $(A,d)$ is a $\op{Q}$-algebra via a unary square-zero operator $\Delta_\op{Q}$.  A trivialization of $\Delta_\op{Q}$ in $(A,d,\Delta_\op{Q})$ permits the lifting of $(A,d)$ from a $\op{Q}$-algebra to a $\Phi_{KMS}(\op{Q})$-algebra.
\end{corollary}
\begin{proof}  Define a map $Hom(\Gamma(\op{Q}),End_A) \to Hom(W(\op{Q}),TEnd_A)$ as follows.  Given $\Gamma(\op{Q})\to End_A$ we define a map of $S^1$-operads $W(\op{Q})\to TEnd_A$ by mapping $\op{Q}$ via said morphism, sending $\Delta\mapsto 0$ and sending $\beta_i$ to the coefficients of the degeneration, as per Lemma $\ref{trivlem}$.  Then $Hom(\Phi_{KMS}(\op{Q}),End_A)\cong Hom(W(\op{Q}),TEnd_A)$ by adjointness, hence the claim.
\end{proof}

Note that although there is a zig-zag of morphisms of dg operads:
\begin{equation*}
\op{Q}\stackrel{\sim}\leftarrow W(\op{Q})\rightarrow \Phi_{KMS}(\op{Q})
\end{equation*}
when we speak of lifting a $\op{Q}$-algebra to a $\Phi_{KMS}(\op{Q})$-algebra we use the corollary (i.e. inclusion), which is not the same as composition in the diagram.

\section{Cyclic cohomology operations.}\label{cycsec}  In this section we give an overview of the mixed complex associated to a multiplicative cyclic operad, its associated spectral sequence, and the accompanying homotopy BV and gravity structures.  The material on cyclic operads and associated complexes comes primarily from \cite{Ward2} sections 2-4.  The material on spectral sequences associated to mixed complexes can be found in \cite{Loday}. 

\subsection{Deformation complexes and mixed complexes from cyclic operads}

Let $\op{O}$ be a unital cyclic operad valued in dg vector spaces graded cohomologically.  We define $\op{O}^\ast:=\prod_n\Sigma\mathfrak{s}\op{O}(n)$, where $\Sigma$ is the degree shift operator and $\mathfrak{s}$ is the operadic suspension.  This vector space has a natural Lie bracket $\{-,-\}$ for which a Maurer-Cartan element is equivalent to a morphism $\mu\colon\op{A}_\infty\to\op{O}$.  Given this data we form a differential in the usual way:  $d_\mu:= d_\op{O}+\{\mu,-\}$, where $d_\op{O}$ is the original (arity-wise) differential in $\op{O}$.

\begin{definition}  Let $\mu\colon\op{A}_\infty\to \op{O}$ be a map of cyclic operads.  Define $CH^\ast(\op{O},\mu)$ to be the complex $(\op{O}^\ast, d_\mu)$.  This complex is called the deformation complex of $\mu$.
\end{definition}

The complex $CH^\ast(\op{O},\mu)$ can be endowed with a square zero operator $\Delta$ of degree $-1$ analogous to the construction of Connes' $B$ operator.  Explicitly $\Delta:=Ns_0(1-t)$ where $s_0$ is the extra degeneracy (via the unit), $t$ is the arity-wise cyclic operator and $N$ is the arity-wise sum $\Sigma_i t^i$   This operator is square zero and commutes with $d_\mu$ and so $(\op{O}^\ast,d_\mu,\Delta)$ is a mixed complex.  See \cite{Ward2} section 3.2 for the details of this construction. 

To this mixed complex we will associate a bicomplex with vertical differential $d_\mu$ (pointing up) and horizontal differential $\Delta$ (pointing right) and filter this bicomplex by columns.  We would like this filtration to be exhaustive, so for simplicity we impose the following 
{\bf assumption:} the complex $\op{O}^\ast$ is supported in non-negative degrees.  This will be the case if $\op{O}(n)$ is contained in degrees $\geq -n$.  Of course there are weaker boundedness conditions under which this filtration is exhaustive.  In this case, we may form the bicomplex in the first quadrant.  We filter this bicomplex by columns and construct the associated spectral sequence.  The $0$-page has $E_0^{pq}=\Sigma^{2p}\op{O}^{q-p}=\prod_{q=0}^\infty\Sigma^{q+2p}\op{O}(n)_{q-p-n}$ the $1$-page has $E_1^{pq}=\Sigma^{2p}H^{q-p}(\op{O}^\ast,d)=\Sigma^{2p}HH^{q-p}(\op{O},\mu)$, and the $2$-page has $E_2^{pq}=H^p(\Sigma^{2\ast}HH^{q-\ast}(\op{O},\mu), [\Delta])$.  In particular, the $1$-page differential is $d^1:=[\Delta]$.  

If $[\Delta]=0$ on $HH$ then there exists a chain operation of degree $-2$ whose boundary is $\Delta$, from which we can build the 2-page differential.  Explicitly, in the notation of the previous section $d^2=\Delta\beta_1$.  Likewise $d^2$ will be zero if there exists a chain operation of degree $-4$, whose boundary is $\Delta\beta_1$, and explicitly one may calculate $d^3=\Delta(\beta_2-\beta_1^2)$.  Continuing in this vein we see existence of the $\beta_i$ is necessary and sufficient for degeneration:

\begin{lemma}(\cite{DSV} Proposition 1.5) Degeneration in this spectral sequence at the $1$-page is equivalent to a trivialization of $\Delta$ in the sense of Definition $\ref{degdef}$.
\end{lemma}

This spectral sequence converges to the product total complex ToT of the original bicomplex.  The homology of the total complex is the cyclic cohomology of the pair $(\op{O}, \mu)$ and will be denoted $HC^\ast(\op{O},\mu)$.  The above lemma is a general statement about mixed complexes, but here we have additional structure allowing us to compute $HC^\ast(\op{O},\mu)$ in two ways.  Namely via the total complex $(\op{O}^\ast[[z]], d+z\Delta)$, where $deg(z)=2$, or by the product complex of invariants $C^\ast_\lambda(\op{O}, \mu):=(\prod_n \op{O}(n)^{\mathbb{Z}_{n+1}}, d_\mu)$.  The complex $C^\ast_\lambda(\op{O}, \mu)$ may also be called the cyclic deformation complex of $\mu$, and it computes $HC^\ast(\op{O},\mu)$ by the following Lemma.

\begin{lemma}\label{cyclem2} Inclusion $C^\ast_\lambda(\op{O}, \mu)\hookrightarrow (\op{O}^\ast[[z]], d+z\Delta)$ is a quasi-isomorphism.
\end{lemma}
\begin{proof}  This follows from Proposition 2.16 of \cite{Ward2} and a standard `killing contractible complexes' argument as in \cite{Loday} 2.4.3.
\end{proof}

In the case that the spectral sequence associated to the mixed complex $(\op{O}^\ast,d_\mu,\Delta)$ degenerates at the $1$-page, we have $HC^\ast(\op{O},\mu)\cong HH^\ast(\op{O},\mu)[[z]]$.  This happens precisely when there is a trivialization of $\Delta$ so we make the following definition:

\begin{definition}\label{degdef2}  We say that a morphism of cyclic operads $\mu\colon\op{A}_\infty\to\op{O}$ is {\it cyclically degenerate} if there exists a trivialization of $\Delta$ in the mixed complex $(\op{O}^\ast,d_\mu,\Delta)$.
\end{definition}

\subsection{Operations on $CH$ and $C_\lambda$}

In the case that $\mu\colon\op{A}_\infty\to\op{O}$ is simply a map of operads (not cyclic) the classical Deligne conjecture says that $CH^\ast(\op{O}, \mu)$ is an algebra over a chain model for the little disks operad.  In the $A_\infty$ case this chain operad was constructed by Kontsevich and Soibelman \cite{KS} and was denoted $\mathsf{M}$ for minimal operad.  It is an insertion operad of planar 2-colored trees.  

If we now remember the cyclic structure of $\mu\colon\op{A}_\infty\to\op{O}$ we may upgrade the complex $CH^\ast(\op{O}, \mu)$ to an algebra over a chain model for the framed little disks.  In the cyclic $A_\infty$ setting this operad was constructed in \cite{Ward1}, where it was denoted $\op{TS}_\infty$.  In hindsight, this $\infty$ notation might be confusing, so we will instead use the notation $\fm$ for `framed minimal operad'.  For details see \cite{Ward1}, but the important properties of the chain operad $\fm$ are:
\begin{enumerate}
\item  There exists a zig-zag of quasi-isomorphisms of operads between $\fm$ and singular chains on the framed little disks.
\item  In particular, there exists a zig-zag of dg operads: $\fm\stackrel{\sim}\leftarrow\op{BV}_\infty\stackrel{\sim}\rightarrow \op{BV}$.
\item  The dg operad $\fm$ acts on $CH^\ast(\op{O}, \mu)$ for any such $\mu\colon\op{A}_\infty\to\op{O}$ recovering the known BV structure on cohomology.
\item  The dg operad $\fm$ is constructed as an insertion operad of planar 2-colored trees along with marked points on the boundaries of vertices, which are called spines.
\end{enumerate} 

We could instead consider the chain model provided by cellular chains on normalized cacti \cite{Vor, KCacti} which has all of the properties above, except requires that the multiplication be associative.

Since the invariants $C_\lambda^\ast(\op{O}, \mu)$ form a $d_\mu$-closed subspace of $CH^\ast(\op{O}, \mu)$ we can ask which of the operations in $\fm$ restrict to this complex.  In other words, which operations preserve the property of $t$-invariance?  This question was answered in \cite{Ward2} and the answer forms a sub-operad $\mathsf{M}_\circlearrowright\subset \fm$ whose homology is the homology operad of the punctured Riemann spheres, called the gravity operad in \cite{Geteq}.  In particular, this endows the cyclic cohomology $HC^\ast(\op{O},\mu)$ with the structure of a gravity algebra and the natural map $HC^\ast(\op{O},\mu)\to HH^\ast(\op{O},\mu)$ is a map of gravity algebras.  It is natural to ask when this structure lifts to the compactification, and this question is answered in the next section.

\section{Collecting Results.}
We shall now collect our main results in this section.  We continue to write $\Phi$ (with suppressed subscript) in place of $\Phi_{k\hookrightarrow H^\ast(S^1)}$.  We define $\fm_{hS^1}:=\Phi_{KMS}(\fm)$.

\begin{theorem}\label{chainmodelthm}  The dg operad $\fm_{hS^1}$ is a chain model for the hyper-commutative operad.  Explicitly, there exists a zig-zag of weak equivalences of dg operads $\fm_{hS^1}\stackrel{\sim}\leftarrow \dots \stackrel{\sim}\rightarrow S_\ast(\overline{\op{M}}_{\ast+1})$.   
\end{theorem}
\begin{proof}  Here $S_\ast$ denotes singular chains with coefficients in the field $k$.  Using formality of  the Deligne-Mumford operad $\overline{\op{M}}_{\ast+1}$, (see \cite{GNPR} Corollary 7.2.1) we know there exists a zig-zag of weak equivalences connecting $\op{H}y\op{C}om\sim S_\ast(\overline{\op{M}}_{\ast+1})$. From \cite{KMS} we know there is a weak equivalence $\op{H}y\op{C}om\stackrel{\sim}\to\Phi_{KMS}(BV)$.  So it remains to show that there is a zig-zag of weak equivalences $\tilde{\Phi}(BV) \sim \tilde{\Phi}(\fm)$.

The algebra $H^\ast(S^1)$ has a Koszul resolution given by taking the cobar construction $\Omega$ of the coalgebra $H_\ast(BS^1)$.  Let $\iota\colon k\hookrightarrow \Omega H_\ast(BS^1)$ be inclusion and let $f\colon\Omega H_\ast(BS^1)\to H^\ast(S^1)$ be the standard weak equivalence.  Then, since $f\circ \iota$ is the standard injection $k\hookrightarrow H^\ast(S^1)$ we endeavor to show $\tilde{\Phi}_{f\circ \iota}(\op{BV}) \sim \tilde{\Phi}_{f\circ \iota}(\fm)$.  Note that $\fm$ is equivalent to $\op{BV}$ in the category under $\Omega H_\ast(BS^1)$.  In particular, there is a zig-zag of weak equivalences $\fm\stackrel{\sim}\leftarrow \op{BV}_\infty\stackrel{\sim}\rightarrow \op{BV}$ under $\Omega H_\ast(BS^1)$.  Therefore $\tilde{\Phi}_\iota(\Gamma_f \fm)\sim \tilde{\Phi}_\iota(\Gamma_f \op{BV})$.  Applying Lemma $\ref{welem2}$ proves the claim.  
\end{proof}
We let $\op{H}y\op{C}om_\infty$ denote the cobar resolution via the Koszul dual operad $\op{G}rav$, i.e. $\op{H}y\op{C}om_\infty:= \Omega \op{G}rav^\ast$.  Cofibrancy of this operad gives us:
\begin{corollary}\label{HCinfcor} There exists a zig-zag of weak equivalences: $\fm_{hS^1}\stackrel{\sim}\leftarrow \op{H}y\op{C}om_\infty\stackrel{\sim}\rightarrow \op{H}y\op{C}om$.
\end{corollary}

We now consider the effect that degeneration has upon the algebraic operations on $CH^\ast(\op{O},\mu)$ and $C_\lambda^\ast(\op{O},\mu)$ and their cohomologies.

\begin{theorem}\label{algthm}  Let $\mu\colon\op{A}_\infty\to\op{O}$ be a morphism of cyclic dg operads which is cyclically degenerate (Definition $\ref{degdef2}$).  Then:
\begin{enumerate}
\item  The $\fm$-algebra structure on $CH^\ast(\op{O},\mu)$ lifts along the inclusion $\fm \to \fm_{hS^1}$ to an $\fm_{hS^1}$-algebra, and hence a $\op{H}y\op{C}om_\infty$-algebra.
\item  There is a $\op{H}y\op{C}om_\infty$-algebra structure on $C_\lambda^\ast(\op{O},\mu)$ for which the inclusion $C_\lambda^\ast(\op{O},\mu)\to CH^\ast(\op{O},\mu)$ extends to an $\infty$-morphism between these $\op{H}y\op{C}om_\infty$-algebra structures.
\item  The gravity structure on $HC^\ast(\op{O},\mu)$ vanishes.
\end{enumerate}
\end{theorem}
\begin{proof} The first statement follows from Corollary $\ref{KMScor}$ and Corollary $\ref{HCinfcor}$.  

For the second statement, we first use the obvious inclusion and projection maps $in\colon(\op{O}^\ast,d_\mu)\leftrightarrows (\op{O}^\ast[[z]], d_\mu)\colon \pi$, with $\pi\circ in=id$, to furnish a morphism of operads $End_{(\op{O}^\ast,d_\mu)}\to End_{(\op{O}^\ast[[z]], d_\mu)}$.  Observe that $\pi$ is a morphism of $\fm_{hS^1}$-algebras, and hence of $\op{H}y{C}om_\infty$-algebras, with respect to the induced $z$-linear extension.   

Under the degeneration hypothesis we have the commutative diagram:
\begin{equation*}
\xymatrix{ C_\lambda^\ast(\op{O}, \mu)  \ar[d]^\sim \ar[r]^{in} & (\op{O}^\ast,d_\mu) \\ (\op{O}^\ast[[z]], d_\mu+\Delta z) \cong (\op{O}^\ast[[z]], d_\mu) \ar[ur]^\pi & }
\end{equation*}
where the isomorphism is given by applying Lemma $\ref{trivlem}$ to the mixed complex $(\op{O}^\ast,d_\mu,\Delta)$.  Since $\pi$ is a morphism of $\op{H}y\op{C}om_\infty$-algebras we can, via a standard transfer argument (see \cite{BMo} Thm 3.5 or \cite{LV} Ch.10) endow $C_\lambda^\ast(\op{O}, \mu)$ with the structure of a $\op{H}y\op{C}om_\infty$-algebra such that the weak equivalence $\downarrow$ extends to an $\infty$-quasi-isomorphism of such.  Composition in the diagram proves statement 2.
 
For statement 3, one sees from \cite{Ward2} that the gravity bracket $g_n$ on $HC^\ast(\op{O},\mu)$, induced via the chain level action of $\mathsf{M}_\circlearrowright $, satisfies the Chas-Sullivan formula (after \cite{CS} p.21):  

\begin{equation} g_n([a_1]\cdc [a_n]) = B(I([a_1])\bullet...\bullet I([a_n])) 
\end{equation}     

where

\begin{equation*}
\dots\to HC^{m-1}(\op{O},\mu) \stackrel{S}\to HC^{m+1}(\op{O},\mu) \stackrel{I}\to HH^{m+1}(\op{O},\mu) \stackrel{B}\to HC^m(\op{O},\mu)\to\dots
\end{equation*}     
is the cyclic cohomology long exact sequence and $\bullet$ is the commutative product.  Since $\mu$ is assumed to be cyclically degenerate, we know that on cohomology $B=0$, from which the claim follows.
\end{proof}

From statement 3 of the theorem we derive the following consequence, which is a piece with the deformation theoretic interpretation of hyper-commutativity (via the Quantum cup product).

\begin{corollary}  The gravity brackets in $End_{HC^\ast(\op{O},\mu)}$ are obstructions to $\mu$ being cyclically degenerate.
\end{corollary}

\begin{remark}  If we consider the usual sequence of operads $\op{G}rav\to\op{BV}\to\op{H}y\op{C}om$ and note that composition in this sequence is the zero map, then we may choose to interpret statement 3 of Theorem $\ref{algthm}$ as a lifting statement.  Indeed it says that under degeneration, the gravity structure on $HC^{\ast}(\op{O},\mu)$ lifts via this sequence to a hyper-commutative algebra.
	
	It would be interesting to prove a chain level refinement of this lifting statement asserting that under degeneration, a $\op{G}rav_\infty$ structure on $C_\lambda^\ast(\op{O}, \mu)$  
	 induced by  $\mathsf{M}_\circlearrowright$ lifts (up to an $\infty$-quasi-isomorphism) to the $\op{H}y\op{C}om_\infty$-algebra structure given in statement 2 of the theorem.  A first step in this direction is to show the existence of a weak equivalence of dg operads $\op{G}rav_\infty\stackrel{\sim}\to\mathsf{M}_\circlearrowright$, a result which will appear as part of upcoming joint work with R. Campos.
	
\end{remark}

\section{Topology and combinatorics of the chain model $\fm_{hS^1}$.}\label{combsec}

The chain models $\mathsf{M}_\circlearrowright, \mathsf{M}$, and $\fm$ may be described using the combinatorics of trees.  We now briefly describe such an interpretation of the chain model $\fm_{hS^1}$.  To begin, we recall that $\fm$ is described via rooted, planar, black and white trees with spines, see Figure $\ref{fig:rect}$.  In particular, the figure depicts the set of cells which ensure the BV equation holds up to homotopy.  Here we have suppressed associahedra labels of black vertices for simplicity, see \cite{Ward1} for details.

\begin{figure}
	\centering
		\includegraphics[scale=1.9]{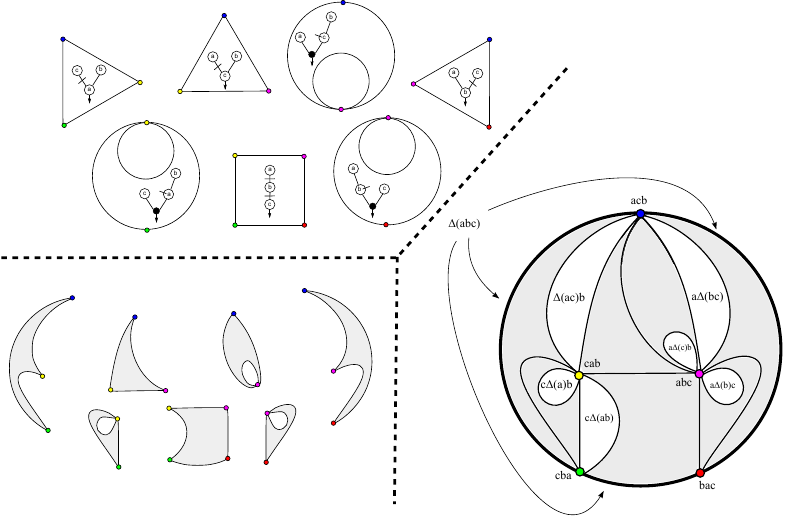}
	\caption{The BV equation on a sphere via $\fm$.  Here we see one hemisphere and 6 holes.  The 7th hole is the missing hemisphere.  If an $\fm$-algebra lifts to an $\fm_{hS^1}$-algebra we may fill in the 7 holes to form a ternary cycle corresponding to the fundamental class of $\overline{\op{M}}_4$.}
	\label{fig:rect}
\end{figure}

Operadic composition at the $\beta_i$ is free, so a general operation in $\fm_{hS^1}$ may be depicted as a (non-planar) rooted tree with tails, each of whose vertices are labeled by an arity appropriate 2-colored planar rooted trees with spines, or by $\beta_i$.  Since the arity of each of the $\beta_i$ is $1$, this is the same thing as a (composition of) planar tree(s) with three types of vertices; the black and white from $\fm$ and the $\beta_i$, see Figure $\ref{fig:trees}$.  To distinguish these new vertices they will be drawn as gray rectangles.

\begin{figure}
	\centering
		\includegraphics[scale=.6]{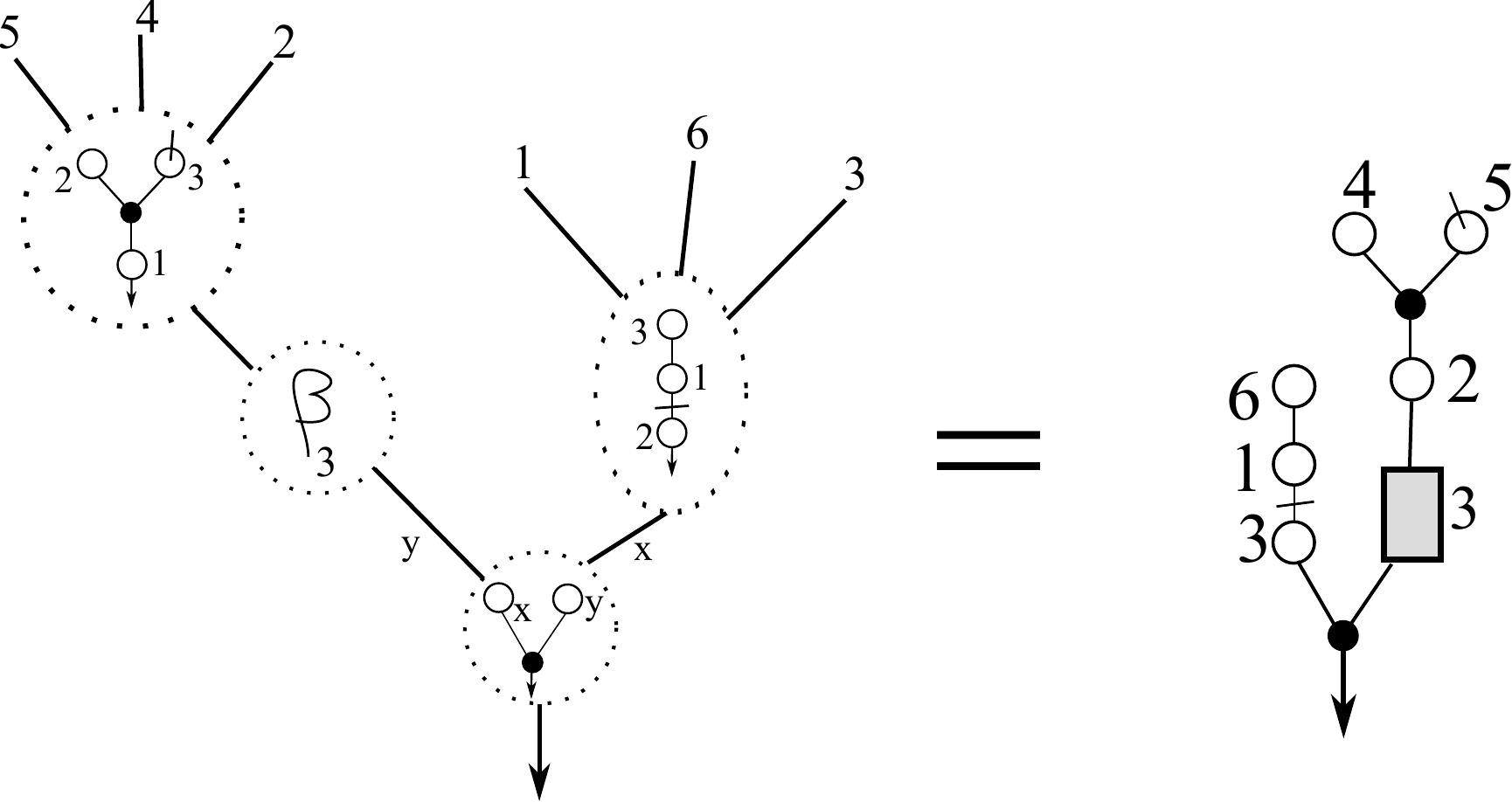}
		\caption{A generator of $\fm_{hS^1}$ as a 3-colored tree on the right hand side.  In general such a correspondence uses the operad composition in $\fm$ and can produce a sum of $3$-colored trees if composing at white vertices of sub-maximum height.}
	\label{fig:trees}
\end{figure}

In particular, we identify $\fm_{hS^1}(n)$ with the span of planar black, white, and gray rooted trees having $n$ white vertices, along with the appropriate vertex labels: white vertices carry labels by $\{1\cdc n\}$ and spines, black vertices carry associahedra cells of appropriate arity as labels, and gray vertices are each labeled by a natural number.  These trees are subject to combinatorial restrictions including all black vertices have at least 2 incoming edges and all gray vertices have exactly 1 incoming and 1 outgoing (half) edge (allowed to be the root).
Notice these trees to not carry tails (unmatched half-edges) except the root.  Also notice that the gray vertices do not contribute to the arity.

The differential can be described as a sum over the vertices and works as in $\fm$ away from the gray vertices.  At gray vertices the label is reduced by $1$ and we take $\Delta$ of the input. For example:
\begin{center}
\includegraphics[scale=.5]{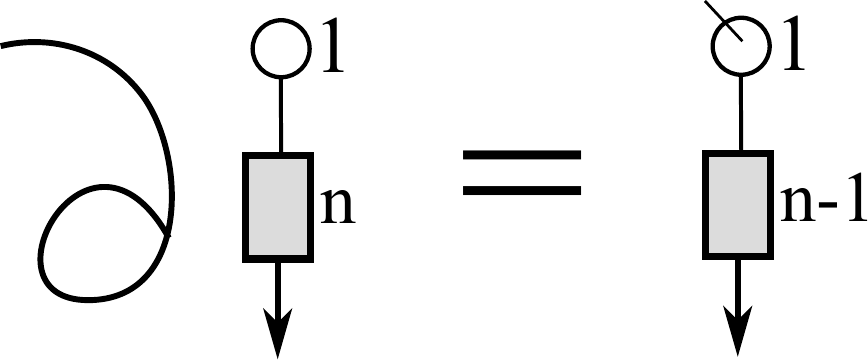}
\end{center}

Consequently, $\fm_{hS^1}(3)$ includes the $2$-cells needed to fill in the holes in the sphere corresponding to the BV equation in Figure $\ref{fig:rect}$.  For example we fill in the hole labeled by $\Delta(ac)b$ with:
\begin{center}
		\includegraphics[scale=.7]{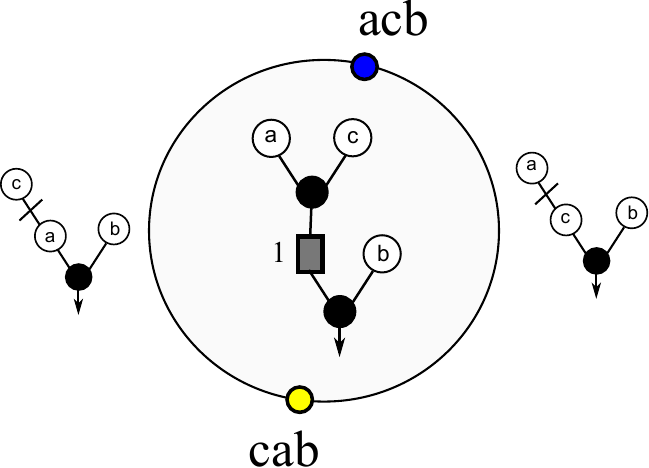}
\end{center}

The sum of these fourteen cells represents the fundamental class of $\overline{\op{M}}_4$ in $\fm_{hS^1}(3)$.  We could similarly describe the fundamental class of $\overline{\op{M}}_5$ by adding the cells prescribed in \cite{LS} to the failure of the appropriate 4-ary relation.

To conclude, let us recall that each chain complex $\fm(n)$ may be viewed as the cell complex of a CW complex \cite{Ward1}.  We may also consider such a description for $\fm_{hS^1}$.  Cellularly $\beta_n$ corresponds to an open $2n$ disk bounded by the $2n-1$ sphere $\overline{\beta_{n-1}\Delta}$ and we may use composition in Cacti to form a CW complex whose cellular chain complex is:
\begin{equation*}
\dots\to k_{\beta_1\Delta}\stackrel{0}\to k_{\beta_1}\stackrel{1}\to k_{\Delta}\stackrel{0}\to k_{id}
\end{equation*}
For example, if we consider $Int(\beta_1\Delta)=Int(D_2\times I)$, then the boundary of the disk should be attached to a composition in Cacti; recall that this composition adds the angles \cite{Vor}.  We thus attach the cell $D_2\times I$ to the closed pointed 2-disk via $(p,0)\mapsto p$ and $(p,1)\mapsto p$ for $p\in Int(D_2)$ and $(e^{2\pi i\theta_1},\theta_2)\mapsto e^{2\pi i(\theta_1+\theta_2)}$ for $e^{2\pi i\theta_1}\in \partial D_2$.  This 3 skeleton is identified with $S^3=\{(r_1e^{i\theta_1},r_2e^{i\theta_2}) : r_1^2+r^2_2=1\}$ by
the map $(r_1e^{i\theta_1},r_2e^{i\theta_2})\mapsto (r_1e^{2\pi i(\theta_1-\theta_2)}, \theta_2+r_1\theta_1)$.  In particular, the cell $\Delta$ is the fiber direction; i.e. concatenating the above map with projection to the first factor and then identifying the unit disk mod boundary with $\mathbb{C}\cup \infty$\ via $p\mapsto p/\sqrt{1-|p|^2}$ we find $(r_1e^{i\theta_1},r_2e^{i\theta_2})\mapsto r_1/r_2e^{2\pi i (\theta_1-\theta_2)}$, the Hopf map.

It is not expected that these CW complexes form a topological operad on the nose, but rather a weak topological operad, due to normalization issues, as in \cite{KCacti}.

\bibliography{cycbib}
\bibliographystyle{alpha}

\end{document}